\documentclass[a4paper,10pt]{article}
\usepackage{mathtext}
\usepackage[T1,T2A]{fontenc}
\usepackage[cp1251]{inputenc}
\usepackage[english]{babel}
\usepackage{amsmath}
\usepackage{amsfonts}
\usepackage{amssymb}
\usepackage{mathrsfs}
\usepackage{amsthm}
\usepackage{textcomp}
\usepackage{euscript}

\newcommand{\eps}{\varepsilon}

\DeclareMathOperator{\Ker}{Ker}

\newcommand{\Bell}{\boldsymbol{B}}

\DeclareMathOperator{\E}{\mathbb{E}}
\DeclareMathOperator{\M}{\EuScript{M}}

\newtheorem{Def}{Definition}
\newtheorem{Rem}{Remark}

\renewcommand{\leq}{\leqslant}
\renewcommand{\geq}{\geqslant}

\newcommand{\gengrad}{\ensuremath{\mathaccent\cdot\nabla}}
\theoremstyle{plain}\newtheorem{Th}{Theorem}

\theoremstyle{plain}\newtheorem{Cj}{Conjecture}
\theoremstyle{plain}\newtheorem{Ex}{Example}
\theoremstyle{plain}\newtheorem{Le}{Lemma}

\begin{document}
\title{Anisotropic Ornstein non inequalities}
\author{Krystian~Kazaniecki \and Dmitriy~M.~Stolyarov\thanks{Supported by RFBR grant no. 14-01-00198.} \and {Michal Wojciechowski}}

\maketitle
\begin{abstract}
We investigate existence of a priori estimates for differential operators in~$L^1$ norm: for anisotropic homogeneous differential operators $T_1, \ldots , T_{\ell}$, we study the conditions under which the inequality
\begin{equation*}
\|T_1 f\|_{L_1(\mathbb{R}^d)} \lesssim \sum\limits_{j = 2}^{\ell}\|T_j f\|_{L_1(\mathbb{R}^d)}
\end{equation*}
holds true.
We also discuss a similar problem for martingale transforms.  
\end{abstract}
\section{Introduction}\label{S1}
In his seminal paper~\cite{Ornstein}, Ornstein proved the following: let~$\{T_j\}_{j=1}^{\ell}$ be homogeneous differential operators of the same order in~$d$ variables (with constant coefficients); if the inequality
\begin{equation*}
\|T_1 f\|_{L_1(\mathbb{R}^d)} \lesssim \sum\limits_{j = 2}^{\ell}\|T_j f\|_{L_1(\mathbb{R}^d)}
\end{equation*}
holds true for any~$f \in C_0^{\infty}(\mathbb{R}^d)$, then~$T_1$ can be expressed as a linear combination of the other~$T_j$. Here and in what follows~``$a \lesssim b$'' means ``there exists a constant~$c$ such that~$a \leq cb$ uniformly'', the meaning of the word ``uniformly'' is clear from the context. For example, in the statement above, the constant should be uniform with respect to all functions~$f$. The aim of the present paper is to extend this theorem to the case where the differential operators are anisotropic homogeneous; see also~\cite{KW}, where a partial progress in this direction was obtained by a simple Riesz product technique. 

To formulate the results, we have to introduce a few notions. Each differential polynomial~$P(\partial)$ in~$d$ variables has a Newton diagram which matches a set of integral points in~$\mathbb{R}^d$ to each such polynomial. The monomial~$a\partial_1^{m_1}\partial_2^{m_2}\ldots\partial_{d}^{m_d}$ corresponds to the point~$m=(m_1,m_2,\ldots, m_d)$; for an arbitrary polynomial, its Newton diagram is the union of the Newton diagrams of its monomials. 

Let~$\Lambda$ be an affine hyperplane in~$\mathbb{R}^d$ that intersects all the positive semi-axes. We call such a plane a \emph{pattern of homogeneity}. We say that a differential polynomial is homogeneous with respect to~$\Lambda$ (or simply~$\Lambda$-homogeneous) if its Newton diagram lies on~$\Lambda$.
\begin{Cj}\label{MainTheorem}
Let~$\Lambda$ be a pattern of homogeneity in~$\mathbb{R}^d$\textup, let~$\{T_j\}_{j=1}^{\ell}$ be a collection of~$\Lambda$-homogeneous differential operators. If the inequality
\begin{equation*}
\|T_1 f\|_{L_1(\mathbb{R}^d)} \lesssim \sum\limits_{j = 2}^{\ell}\|T_j f\|_{L_1(\mathbb{R}^d)}
\end{equation*}
holds true for any~$f \in C_0^{\infty}(\mathbb{R}^d)$\textup, then~$T_1$ can be expressed as a linear combination of the other~$T_j$.
\end{Cj}  
This conjecture may seem to be a simple generalization of Ornstein's theorem. We warn the reader that sometimes the anisotropic character of homogeneity brings new difficulties to inequalities for differential operators (the main is that one lacks geometric tools such as the isoperimetric inequality, or the coarea formula, etc.). For example, the classical embedding~$W_1^1(\mathbb{R}^d) \hookrightarrow L_{\frac{d}{d-1}}$ due to Gagliardo and Nirenberg had been generalized to the anisotropic case only in~\cite{Solonnikov} and finally in~\cite{Kolyada}; if one deals with similar embeddings for vector fields, the isotropic case was successfully considered in~\cite{vS} (see also the survey~\cite{vS2}), and there is almost no progress for anisotropic case (however, see~\cite{KMSshort,KMS}).


The method we use to attack the conjecture, differs from that of Ornstein (though there are some similarities). However, it is not new. It was noticed in~\cite{CFM} that Ornstein's theorem is related to the behavior of certain rank one convex functions. The case~$d=2$ was considered there. As for the general case of Ornstein's (isotropic) theorem, its proof via rank one convexity was announced in~\cite{KK} (and the proofs are now available in a very recent preprint~\cite{KK2}). In a sense, we follow the plan suggested in~\cite{KK}. However, the notions of quasi convexity, rank one convexity and others should be properly adjusted to the anisotropic world, we have not seen such an adjustment anywhere. For all these notions in the classical setting of the first gradient, their relationship with each other, properties, etc., we refer the reader to the book~\cite{Dacorogna}. 
The approach of rank one convexity reduces Conjecture~\ref{MainTheorem} to a certain geometric problem about separately convex functions (Theorem~\ref{SeparateTheorem}) that is covered by Theorem 1 announced in~\cite{KK} (Theorem~$1.1$ in~\cite{KK2}). We give a simple proof of this fact, which may seem the second advantage of our paper (though our proof does not give more advanced Theorem 1 of~\cite{KK}). We did not know the preprint~\cite{KK2} almost until the publication of the present text, and did our work independently. Though the spirit of our approach in the geometric part is similar to that of~\cite{KK2}, the presentation and details appear to be different. 

We will prove a particular case of Conjecture~\ref{MainTheorem}, which still seems to be rather general (in particular, it covers the classical isotropic case).
\begin{Th}\label{MainTheoremBis}
Let~$\Lambda$ be a pattern of homogeneity in~$\mathbb{R}^d$\textup, let~$\{T_j\}_{j=1}^{\ell}$ be~$\Lambda$-homogeneous differential operators. Suppose that all the monomials present in the~$T_j$ have one and the same parity of degree. If the inequality
\begin{equation}\label{MainNonInequality}
\|T_1 f\|_{L_1(\mathbb{R}^d)} \lesssim \sum\limits_{j = 2}^{\ell}\|T_j f\|_{L_1(\mathbb{R}^d)}
\end{equation}
holds true for any~$f \in C_0^{\infty}(\mathbb{R}^d)$\textup, then~$T_1$ can be expressed as a linear combination of the other~$T_j$.
\end{Th}
We note that the differential operators here are not necessarily scalar, i.e., one can prove the same theorem for the case where operators act on vector fields. It is one of the advantages of the general rank one convexity approach. However, to facilitate the notation, we work on the scalar case.

We outline the structure of the paper. We begin with restating inequality~\eqref{MainNonInequality} as an extremal problem described by a certain Bellman function (if inequality~\eqref{MainNonInequality} holds, then the corresponding Bellman function is non-negative).
We also study the properties of our Bellman function (they are gathered in Theorem~\ref{PropertiesBell}), the most important of which is the quasi convexity. All this material constitutes Section~\ref{S2}. It turns out, that quasi convexity leads to a softer, but easier to work with, property of rank one convexity. The proof of this fact is given in Section~\ref{S3}, see Theorem~\ref{RankOne}. So, the Bellman function in question is rank one convex. In Section~\ref{S4}, we prove that rank one convex functions homogeneous of order one are non-negative, which gives us Theorem~\ref{MainTheoremBis}. In fact, it suffices to show a similar principle for separately convex functions on~$\mathbb{R}^d$, which is formalized in Theorem~\ref{SeparateTheorem}. This theorem is purely convex geometric. Finally, we discuss related questions in Section~\ref{S5}. 

We thank Fedor Nazarov and Pavel Zatitskiy for fruitful discussions on the topic, and Alexander Logunov for his help with subharmonic functions.

\section{Bellman function and its properties}\label{S2}
Inequality~\eqref{MainNonInequality} can be rewritten as
\begin{equation}\label{Infimum}
\inf_{\varphi \in C_0^{\infty}([0,1]^d)} \Big(\sum\limits_{j = 2}^{\ell}\|T_j \varphi\|_{L_1(\mathbb{R}^d)} - c\|T_1 \varphi\|_{L_1(\mathbb{R}^d)}\Big) = 0,
\end{equation}
where~$c$ is a sufficiently small positive constant.
\begin{Def}
Suppose that~$\partial^{\alpha}$\textup,~$\alpha \in A$ are all the partial derivatives that are present in the~$T_j$ \textup(thus~$A$ is a subset of~$\Lambda \cap \mathbb{Z}^d$\textup). Consider the Hilbert space~$E$ with an orthonormal basis~$e_{\alpha}$ indexed with the set~$A$. For each function~$\varphi$ and each point~$x$\textup, we have a mapping
\begin{equation*}
[0,1]^d \ni x \mapsto \gengrad[\varphi](x) = \sum\limits_{\alpha \in A} \partial^{\alpha}[\varphi](x) e_{\alpha} \in E.
\end{equation*}  
We call the function $\gengrad[\varphi]$ the generalized gradient of~$\varphi$. 
\end{Def}
The operator $\gengrad[\cdot]$ is an analogue of the usual gradient suitable for our problem.
\begin{Ex}
 Let $T_j= \partial_{x_j}$ for $j=1, \ldots , d$. In this case the generalized gradient turns out to be the usual gradient on the Euclidean space $\mathbb{R}^d$.
\end{Ex}
\begin{Ex}
Let us take the differential operators
 \begin{equation}
 \begin{aligned}
  T_1 [\varphi] = \partial^{(2,0,1)}[\varphi] - \partial^{(0,3,1)}[\varphi],\quad  T_2 [\varphi] = \partial^{(4,0,0)}\varphi,\quad T_3 [\varphi] = \partial^{(0,6,0)}[\varphi], \quad  T_4 [\varphi] = \partial^{(0,0,2)}[\varphi].
  \end{aligned}
 \end{equation}
We can list all the partial derivatives present in the operators\textup{:}  $$A=\{ \partial^{(0,0,2)}, \partial^{(0,6,0)}, \partial^{(4,0,0)}, \partial^{(0,3,1)}, \partial^{(2,0,1)}\}.$$ All the operators $T_j$ are $\Lambda$-homogeneous\textup, where $\Lambda=\{ x\in R^{3} \colon \langle x,\, (3,2,6)\rangle = 12\}$.
In this case the generalized gradient is of the following form\textup{:}
\begin{equation*}
\gengrad [\varphi] = (\partial^{(0,0,2)} [\varphi], \partial^{(0,6,0)}[\varphi], \partial^{(4,0,0)}[\varphi], \partial^{(0,3,1)}[\varphi], \partial^{(2,0,1)}[\varphi])\in\mathbb{R}^5.
\end{equation*}
 \end{Ex}
We also consider the function~$V:E\to \mathbb{R}$ given by the rule
\begin{equation}\label{FunctionV}
V(e) = \Big(\sum\limits_{j = 2}^{\ell}|\tilde{T}_j e| - c|\tilde{T}_1 e|\Big), 
\end{equation}
here~$\tilde{T}_j$ are the linear functionals on~$E$ such that~$\tilde{T}_j(e) = \sum_A c_{\alpha,j}e_{\alpha}$ if~$T_j = \sum_{A}c_{\alpha,j}\partial^{\alpha}$. 
With this portion of abstract linear algebra, we rewrite formula~\eqref{Infimum} as
\begin{equation*}
\inf_{\varphi \in C_0^{\infty}([0,1]^d)} \int\limits_{[0,1]^d} V(\gengrad[\varphi](x))\,dx = 0.
\end{equation*}
The main idea is to consider a perturbation of this extremal problem, i.e., the function~$\Bell:E \to \mathbb{R}$ given by the formula
\begin{equation}\label{Bellman}
\Bell(e) = \inf_{\varphi \in C_0^{\infty}([0,1]^d)} \int\limits_{[0,1]^d} V(e+\gengrad[\varphi](x))\,dx.
\end{equation}
\begin{Th}\label{PropertiesBell}
Suppose that inequality~\textup{\eqref{Infimum}} holds true. Then\textup, the function~$\Bell$ possesses the properties listed below.
\textup{
\begin{enumerate}
\item\emph{It satisfies the inequalities~$-\|e\| \lesssim \Bell(e) \lesssim \|e\|$ and~$\Bell \leq V$.}
\item\emph{It is positively one homogeneous\textup, i.e.~$\Bell (\lambda e) = |\lambda|\Bell(e)$.}
\item\emph{It is a Lipschitz function.}
\item\emph{It is a generalized quasi convex function\textup, i.e. for any~$\varphi \in C_0^{\infty}([0,1]^d)$ and any $e\in E$  the inequality}
\begin{equation}\label{BellmanInequality}
\Bell(e) \leq \int\limits_{[0,1]^d} \Bell(e + \gengrad [\varphi](x))\,dx
\end{equation}
\emph{holds true}.
\end{enumerate}
}
\end{Th}
\begin{proof}
1) We get the upper estimates on the function $\Bell$ by plugging $\varphi\equiv 0$ in the formula for it: 
\begin{equation*}
 \Bell(e) \leq \int\limits_{[0,1]^d} V(e+\gengrad[\varphi])= V(e) \lesssim \|e\|. 
\end{equation*}
We obtain the lower bounds on the function $\Bell$ from inequality {\eqref{Infimum}} and the triangle inequality: 
\begin{equation*}
\begin{split}
&\int\limits_{[0,1]^d} \left(\sum\limits_{j = 2}^{\ell}\left|\tilde{T}_j \left( e+ \gengrad[\varphi]\right)\right| - c\left|\tilde{T}_1\left( e+\gengrad[\varphi]\right)\right|\right)
\geq \int\limits_{[0,1]^d}\left( \sum\limits_{j = 2}^{\ell}\left|\tilde{T}_j\left( e+ \gengrad[\varphi]\right)\right| - c\left|\tilde{T}_1 \left(\gengrad[\varphi]\right)\right| - c |\tilde{T}_1 e|\right)\geq
\\&\int\limits_{[0,1]^d} \left(\sum\limits_{j = 2}^{\ell}|\tilde{T}_j (e+\gengrad[\varphi])| -  \sum\limits_{j = 2}^{\ell}|\tilde{T}_j (\gengrad[\varphi])| - c |\tilde{T}_1 e|\right)
=\int\limits_{[0,1]^d}\left( \sum\limits_{j = 2}^{\ell}\left(|\tilde{T}_j \left(e+\gengrad[\varphi]\right)| -  |\tilde{T}_j (\gengrad[\varphi])|\right) - c |\tilde{T}_1 e|\right)\geq
\\&\hspace*{12cm}  - \sum\limits_{j = 2}^{\ell}|\tilde{T}_j e| - c |\tilde{T}_1 e|,
\end{split}
\end{equation*}
where~$\varphi \in C_0^{\infty}([0,1]^d)$~is an arbitrary function.
We take infimum of the above inequality over all admissible~$\varphi$:
\begin{equation*}
 -\|e\|\lesssim   - \sum\limits_{j = 2}^{\ell}|\tilde{T}_j e| - c |\tilde{T}_1 e| \leq \Bell(e).
\end{equation*}
 2) Since $V$ is a positively one homogeneous function, the following equality holds for every $\lambda\neq 0$:
\begin{equation*}
\Bell(\lambda e) = \inf_{\varphi \in C_0^{\infty}([0,1]^d)} \int\limits_{[0,1]^d} V\left(\lambda e+\gengrad[\varphi]\right)
= \inf_{\varphi \in C_0^{\infty}([0,1]^d)} \int\limits_{[0,1]^d}|\lambda| V\left(e+\gengrad[\lambda^{-1}\varphi]\right).
\end{equation*}
We know that $\lambda^{-1}C_0^{\infty}([0,1]^d)=C_0^{\infty}([0,1]^d)$ for every $\lambda\neq 0$, therefore
\begin{equation*}
\Bell(\lambda e)=\inf_{\varphi \in C_0^{\infty}([0,1]^d)} \int\limits_{[0,1]^d}|\lambda|V\left(e+\gengrad[\lambda^{-1}\varphi]\right)=|\lambda| \inf_{\varphi \in C_0^{\infty}([0,1]^d)} \int\limits_{[0,1]^d}V\left(e+\gengrad[\varphi]\right)=|\lambda| \Bell(e).
\end{equation*}
3) In order to get the Lipschitz continuity of $\Bell$, we rewrite the formula for it:
\begin{equation*}
\forall e\in E \qquad \Bell(e)= \inf_{\varphi \in C_0^{\infty}([0,1]^d)} V_{\varphi}(e), 
\end{equation*}
where
\begin{equation*}
  V_{\varphi}(e)= \int\limits_{[0,1]^d} V(e+\gengrad[\varphi](x))dx.
\end{equation*}
It follows from the Lipschitz continuity of $V$ that every function $V_{\varphi}$ is a Lipschitz function with the Lipschitz constant bounded by $L$, where $L$ is the Lipschitz constant of the function $V$. For every two points $v_1,v_2 \in E $, we can find a sequence of functions $V_{\varphi_n}$ such that
$\Bell(v_j)=\inf_{n\in\mathbb{N}} V_{\varphi_n}(v_j)$ for $j\in\{1,2\}$. We define $$f_k (e) = \min_{n=1,2,\ldots, k} V_{\varphi_n}(e).$$ For every $k\in\mathbb{N}$ the function $f_k$ is the Lipschitz function with the Lipschitz constant bounded by $L$. Hence
\begin{equation*}
|\Bell(v_1) - \Bell(v_2)|=\lim_{k\rightarrow \infty} |f_k(v_1) - f_k(v_2)|\leq L \|v_1-v_2\|.
\end{equation*}
4) Before we prove the generalized quasi convexity of this function, we need to introduce some notation. 
We know that all $\alpha\in A$ have common pattern of homogeneity $\Lambda$, thus we can find a vector $\gamma \in \mathbb{N}^d$ and a number~$k\in\mathbb{N}$ such that~$\langle \alpha, \gamma \rangle = k$ for every $\alpha\in A$.

For every $\lambda\in \mathbb{R}$ and~$x\in \mathbb{R}^d$ we denote 
\begin{equation*}
 x_{\lambda}= (\lambda^{\gamma_1} x_1, \lambda^{\gamma_2} x_2, \ldots, \lambda^{\gamma_d} x_d).
\end{equation*}
For every $\lambda\in \mathbb{N}$ we define the partition of the unit cube~$[0,1]^d$ into small parallelepipeds:
\begin{equation*}
Q_y= y+ \Pi_{j=1}^d [0,\lambda^{-\gamma_j}] \quad \hbox{for every}\; y\in Y,\quad \hbox{where} 
\end{equation*}
\begin{equation*}
Y=\left\{y\in [0,1]^d : y=\left(\frac{k_1}{\lambda^{\gamma_1}},\frac{k_2}{\lambda^{\gamma_2}},\ldots, \frac{k_d}{\lambda^{\gamma_d}}\right) \quad \mbox{for }k_j\in\mathbb{N}\cup\{0\}\mbox{ and } k_j<\lambda^{\gamma_j}\right\}.
\end{equation*}
Here~$Y$ is the set of ``leftmost lowest'' vertices of the parallelepipeds~$Q_y$.
The parallelepipeds $Q_y$ are disjoint up to sets of measure zero and $\bigcup_{y\in Y} Q_y = [0,1]^d$.
Let us fix $\varphi\in C_0^{\infty}([0,1]^d)$. Since $\gengrad[\varphi]$ is a uniformly continuous function on $[0,1]^d$ and the diameter of the parallelepipeds $Q_y$ tends to zero uniformly with the growth of $\lambda$, we can choose $\lambda$ sufficiently large to obtain 
\begin{equation}\label{uniformcontpsi}
 \forall\; y\in Y\;\;\forall z,v\in Q_y \qquad |\gengrad[\varphi](z) -\gengrad[\varphi](v)|\leq \frac{\varepsilon}{L},
\end{equation}
where $L$ is the Lipschitz constant of the function $V$. 
Let $\{ \psi_{y}\}_{y\in Y}$ be a family of functions in $C_0^{\infty}([0,1]^d)$. For these functions, we use the following rescaling: 
$$\psi_{y,\lambda}(x)= \lambda^{-k}\psi_{y}((x-y)_{\lambda}).$$ 
Let us observe that the rescaling $(x-y)_{\lambda}$ transforms the cube $[0,1]^d$ into $Q_y$, thus $\operatorname{supp}\psi_{y,\lambda}(x)\subset Q_y $.
Moreover,  we know that 
\begin{equation*}
 \partial^{\alpha}[ \psi_{y,\lambda}](x) = \lambda^{-k} \lambda^{(\sum_{j=1}^d \alpha_j \gamma_j)} \partial^{\alpha} [\psi_y] \left( (x-y)_{\lambda}\right)= \partial^{\alpha} [\psi_{y}] \left( (x-y)_{\lambda}\right)
\end{equation*}
for every $\alpha\in A$.
By~\eqref{Bellman}, we have 
\begin{equation*}
\begin{split}
 \Bell(e)\leq  \int\limits_{[0,1]^d} V\Big( e + \sum_{y\in Y} \gengrad[\psi_{y,\lambda}](x) + \gengrad[\varphi](x)\Big)dx = \sum_{y\in Y} \int\limits_{Q_y} V\Big(e + \gengrad[\psi_{y,\lambda}](x) + \gengrad[\varphi](x)\Big)dx.
 \end{split}
 \end{equation*}
We assumed that \eqref{uniformcontpsi} holds, therefore, for arbitrary $v_y\in Q_y$ we have the following estimate:
\begin{equation*}
\begin{split}
 \int\limits_{Q_y} V\left(e + \gengrad[\psi_{y,\lambda}](x) + \gengrad[\varphi](x)\right)\,dx &\leq \int\limits_{Q_y} V\left(e + \gengrad[\psi_{y,\lambda}](x) + \gengrad[\varphi](v_y)\right)dx + \varepsilon |Q_y|
 \\&= \int\limits_{Q_y} V\left(e + \gengrad[\psi_{y}]((x-y)_{\lambda}) + \gengrad[\varphi](v_y)\right)dx + \varepsilon |Q_y|.
\end{split}
 \end{equation*}
Since $\lambda^{-\left(\sum_{j=1}^{d} \gamma_j\right)}= |Q_y|$,  we have 
\begin{equation*}
  \int\limits_{Q_y} V\left(e + \gengrad[\psi_{y}]((x-y)_{\lambda}) + \gengrad[\varphi](v_y)\right)\,dx = |Q_y|\int\limits_{[0,1]^d}  V\left(e + \gengrad[\psi_{y}](z) + \gengrad[\varphi](v_y)\right)dz
\end{equation*}
for $z =(x-y)_{\lambda}$.
Now for every $y\in Y, v_y \in Q_y$ we can choose $\psi_{y}$ such that 
\begin{equation*}
 \int\limits_{[0,1]^d} V\big(e + \gengrad[\psi_{y}](z) + \gengrad[\varphi](v_y)\big)\,dz\leq \Bell(e+\gengrad[\varphi](v_y))  +\varepsilon
\end{equation*}
(this choice depends on~$v_y$, however, we treat~$v_y$ as of a fixed parameter).
We obtain  
\begin{equation*}
 \Bell(e)\leq  \sum_{y\in Y}|Q_y|\Bell(e+\gengrad[\varphi](v_y)) + 2 \varepsilon
\end{equation*}
from the above inequalities.
We take mean integrals of this inequality over each cube $Q_y$ with respect to~$v_y$, which gives us 
\begin{equation*}
  \Bell(e)\leq \sum_{y\in Y} \int_{Q_y}\Bell(e+\gengrad[\varphi](v_y)) dv_y  +2 \varepsilon= \int_{[0,1]^d} \Bell(e+ \gengrad[\varphi](x))dx + 2\varepsilon.
\end{equation*}
Since $\varepsilon$ was an arbitrary positive number, we have proved the generalized quasi convexity of $\Bell$.
\end{proof}
The proof of the fourth point seems very similar to the standard \emph{Bellman induction step} (see~\cite{NTV,Osekowski,SZ,Volberg} or any other paper on Bellman function method in probability or harmonic analysis); moreover, the function~$\Bell$ itself is, in a sense, a Bellman function and inequality~\eqref{BellmanInequality} is a Bellman inequality. We suspect that this ``similarity'' should be more well studied.

\section{Rank one convexity}\label{S3}
Inequality~\eqref{BellmanInequality} looks like a convexity inequality. Sometimes it is really the case.
\begin{Def}\label{RankOneTensor}
We call a vector~$e_x \in E$ a generalized rank one vector if it is of the form
\begin{equation*}
\sum\limits_{\alpha \in A}  i^{|\alpha|+|\alpha_0|}x^{\alpha}e_{\alpha},\quad x \in \mathbb{R}^d,\; \alpha_0\in A.
\end{equation*}
\end{Def}
\begin{Rem}
 In Theorem \textup{\ref{MainTheoremBis}}\textup, we only consider the case where every $\alpha\in A$ has the same parity as the other elements of $A$. Therefore\textup,~$i^{|\alpha|+|\alpha_0|}\in \mathbb{R}$ for every $\alpha_0, \alpha \in A$a . Hence the coefficients of the generalized rank one vector are real.
\end{Rem}
\begin{Th}\label{RankOne}
The function~$\Bell$ is a generalized rank one convex function\textup, i.e. it is convex in the directions of generalized rank one vectors.
\end{Th} 
To prove the theorem, we need two auxiliary lemmas.

\begin{Le}\label{Laminate}
For every~$x\in\mathbb{R}^d$ and every~$\eps,\delta >0$\textup{,} there exists a function~$l_{x,\eps,\delta} \in C_0^{\infty}([0,1]^d)$ and a set $B\subset [0,1]^d$ such that the following holds.
\textup{\begin{enumerate}
 \item $\|\gengrad[l_{x ,\eps,\delta}]\| \leq \|e_x\| + \eps$.
 \item $|B|\geq 1-\delta.$
 \item \emph{The function~$\gengrad [l_{x ,\eps,\delta}]\big|_{B}$ with respect to the measure $\mu=|B|^{-1}dx$ is equimeasurable with the function $\cos(2\pi t) e_x$\textup, $t\in[0,1]$\textup, i.e. }
 \begin{equation*}
 \mu\left( \{\gengrad[l_{x ,\eps,\delta}]\in W\}\right) = \bigg|\bigg\{ t\in[0,1] : \cos(2\pi t)e_x  \in W\bigg\}\bigg|
 \end{equation*}
 \emph{for every Borel set $W$ in $E$.} 
 \end{enumerate}
 }
 \end{Le}
\begin{proof}
For a given $x\in\mathbb{R}^d$ we take the same $\gamma$ and $k$ as in the proof of the fourth point of Theorem \ref{PropertiesBell}. We consider the function 
\begin{equation*}
l_{ x,\eps,\delta} (\xi)=t^{-k}\cos(\sum_{j=1}^d t^{\gamma_j}x_j\xi_j)\Phi(\xi),
\end{equation*}
where $\Phi$ is the smooth hat function:
\begin{equation*}
\Phi(\xi)=\left\{\begin{array}{lr}
                 1&\qquad \xi\in [2\delta',1-2\delta']^d,\\
                 0&\qquad \xi\in [0,1]^d\backslash[\delta',1-\delta']^d,\\
                 \varTheta(\xi)\in[0,1]& \mbox{otherwise}.
                \end{array}\right.
\end{equation*}
for $\delta'$ sufficiently small (in particular, we need~$2(2\delta')^d < \delta$). Similarly to the fourth point of Theorem \ref{PropertiesBell}, we define the set of proper parallelepipeds
\begin{equation*}
 Y_t=\left\{ Q \colon Q = (k_j v_j )_{j=1,\ldots,d} + \Pi_{j=1}^{d} [0,w_j];\, k_j\in \{1\}\cup\big\{k_j \in \mathbb{N}\colon \; k_j< \frac{ t^{\gamma_j}x_j}{2\pi} - 1 \big\}\right\}.
\end{equation*}
where $ v_j=w_j=2\pi t^{-\gamma_j}x_j^{-1}$ if $x_j\neq 0$ and $v_j=\delta'$, $w_j=(1-2\delta')$ otherwise. For any~$\delta'$, we can choose~$t$ to be so large that
\begin{equation*}
\quad \Big|\bigcup_{\genfrac{}{}{0pt}{-2}{Q \in Y_t}{Q \subset [2\delta',1-2\delta']^d}}\!\!\!Q\Big| \geq 1-\delta. 
\end{equation*}
We put $B$ to be this union, i.e. the union of the parallelepipeds~$Q$ from the family~$Y_t$ that belong to~$[2\delta',1-2\delta']^d$ entirely.
 
If~$t$ is sufficiently large, then for every $\beta\in\mathbb{N}^d$ satisfying $0\leq\langle \beta, \gamma\rangle <k$, we have
\begin{equation}\label{smallnesofnablaPhi}
 \sup_{\xi\in [0,1]^d} |t^{-1} \partial^{\beta}[\Phi](\xi)| \leq \eps'.
\end{equation}
For any $\beta\in \mathbb{N}^d$, the following holds: 
\begin{equation*}
 \partial^{\beta}\left[\cos\left(\sum_{j=1}^d\! t^{\gamma_j}x_j\xi_j\right)\right]= \! t^{\langle \beta, \gamma \rangle} x^{\beta}  \partial^{\beta} [\cos] \left (\sum_{j=1}^d t^{\gamma_j}x_j\xi_j\right).
\end{equation*}

Since all $\alpha\in A$ have the same parity, we either have $\partial^{\alpha} [\cos](\xi) =(-1)^{\frac{|\alpha|}{2}}\cos(\xi) $ for every $\alpha\in A$ or $\partial^{\alpha} [\cos] (x)=(-1)^{\frac{|\alpha|+1}{2}}\sin(\xi) $ for every $\alpha\in A$. Without lost of generality we may assume $2\big||\alpha|$, because the functions sine and cosine are equimeasurable on their periodic domains.
Therefore, for every $\xi\in[0,1]^d$ and $\alpha\in A$ we have
\begin{equation}\label{nablal}
\begin{split}
 \partial^{\alpha}[l_{\xi,\eps,\delta}](\xi)&=\Phi (\xi)\partial^{\alpha}[t^{-k}\cos\big(\sum_{j=1}^d t^{\gamma_j}x_j\xi_j\big)]+ \sum\limits_{\substack{\alpha'+\beta=\alpha\\ \beta\neq (0,0,\ldots, 0)}} c_{\alpha', \beta}\, t^{-k}\partial^{\alpha'}\left[\cos\big(\sum_{j=1}^d t^{\gamma_j}x_j\xi_j(x)\big)\right]\partial^{\beta}[\Phi]
 \\&= \Phi(\xi) x^{\alpha} \partial^{\alpha}[\cos]\big(\sum_{j=1}^d t^{\gamma_j}x_j\xi_j\big)  + \sum\limits_{\substack{\alpha'+\beta=\alpha\\ \beta\neq (0,0,\ldots, 0)}} c_{\alpha', \beta}\, t^{\langle \alpha', \gamma \rangle -k}\partial^{\alpha'}[\cos]\big(\sum_{j=1}^d t^{\gamma_j}x_j\xi_j(x)\big)\partial^{\beta}[\Phi]
 \\&=  (-1)^{\frac{|\alpha|}{2}} x^{\alpha}\cos \big(\sum_{j=1}^d t^{\gamma_j}x_j\xi_j\big) + \hbox{error},
\end{split}
\end{equation} 
where the coefficients $c_{\alpha', \beta}$ come from the Leibniz formula. The error is~$O(\eps')$ in absolute value by~\eqref{smallnesofnablaPhi} and equals to zero on the set~$[2\delta',1-2\delta']^d$ (because the function~$\Phi$ is constant there).
For every $\xi\in [0,1]^d$ we have
\begin{equation*}
\begin{split}
  \gengrad[l_{\xi,\eps,\delta}](\xi)=  \sum_{\alpha\in A}\partial^{\alpha}[l_{\xi,\eps,\delta}](\xi)e_{\alpha}& =  \sum_{\alpha\in A}\left( (-1)^{\frac{|\alpha|}{2}} x^{\alpha}\cos (\sum_{j=1}^d t^{\gamma_j}x_j\xi_j) + \hbox{error}\right) e_{\alpha}
  \\&= e_x \cos(\sum_{j=1}^d t^{\gamma_j}x_j\xi_j) + \hbox{error}.
\end{split}
  \end{equation*}
Thus, for every $\xi\in [0,1]^d$ and $\eps'$ sufficiently small, we obtain
\begin{equation*}
  \left\|\gengrad[l_{\xi,\eps,\delta}](\xi)\right\|\leq \|e_x\| +\left\| \hbox{error}\right\|\leq \|e_x\|+  \eps .
\end{equation*}
Since the error equals to zero on the set~$[2\delta',1-2\delta']^d$, it follows from \eqref{nablal} that for every $\xi\in B$ we have
\begin{equation*}
  \gengrad[l_{\xi,\eps,\delta}](\xi)= \cos (\sum_{j=1}^d t^{\gamma_j}x_j\xi_j) e_x.
\end{equation*}
We note that the function $\cos (\sum_{j=1}^d t^{\gamma_j}x_j\xi_j) e_x$ restricted to any~$Q \in Y_t$ is equimeasurable (with respect to the measure~$\frac{dx}{|Q|}$ on~$Q$) with the function $\cos(2\pi t) e_x$,~$t \in [0,1]$, (one can verify this fact using an appropriate dilation). Since~$B$ is a union of several parallelepipeds~$Q$, the same holds with~$Q$ replaced by~$B$.
\end{proof}

\begin{Le}\label{ConvexityWithCosine}
Suppose that~$v:\mathbb{R} \to \mathbb{R}$ is a Lipschitz function such that 
\begin{equation}\label{meanconvex}
v(x) \leq \int\limits_{0}^{1}v(x + \lambda \cos (2 \pi t)) dt
\end{equation}
for any~$x,\lambda \in \mathbb{R}$.
Then\textup,~$v$ is convex.
\end{Le}
\begin{proof}
We are going to verify that~$v$ is convex as a distribution, or what is the same, that the distribution~$v''$ is non-negatvie. For that, we multiply inequality \eqref{meanconvex} by a positive function $\varphi\in C^{\infty}_{0}(\mathbb{R})$. Since $v$ is a Lipschitz function, we can integrate it over $\mathbb{R}$: 
\begin{equation*}
\begin{split}
\int\limits_{\mathbb{R}} v(x)\phi(x) dx 
&\leq\int\limits_{\mathbb{R}} \int\limits_{0}^{1}v(x + \lambda \cos (2 \pi t))\varphi(x) dt dx
=\int\limits_{\mathbb{R}} \int\limits_{0}^{1}v(x)\varphi(x-\lambda \cos (2 \pi t)) dt dx
\\&=\int\limits_{\mathbb{R}} v(x)\int\limits_{0}^{1}\left( \varphi(x) - \lambda\cos(2\pi t) \varphi^{'} (x) +\frac{\lambda^2}{2}\cos^2 (2\pi t) \varphi^{''}(x) + o(\lambda^2)\right) dt dx
\\&=\int\limits_{\mathbb{R}}\left( v(x)\varphi(x) + v(x)\varphi^{''}(x) \frac{\lambda^2}{2} \left(\int\limits_{0}^{1}\cos^2 (2\pi t)\right) + o(\lambda^2)\right) dx.
\end{split}
\end{equation*}
Therefore,
\begin{equation*}
 0\leq\frac{1}{2}\left(\int\limits_{0}^{1}\cos^2 (2\pi t)dt\right)\int\limits_{\mathbb{R}} v(x)\phi^{''}(x)dx + \frac{o(\lambda^2)}{\lambda^2}.
\end{equation*}
Letting~$\lambda \to 0$, we show that $v^{''}$ as a distribution satisfies $v^{''}(\phi)\geq 0$ for all $\phi\in C^{\infty}_{0}(\mathbb{R})$ and  $\phi\geq 0$.
From the Schwartz theorem it follows that $v^{''}$ is a non negative measure of locally finite variation. Thus $v'$ is an increasing function and therefore $v$ is convex.
\end{proof}
\paragraph{Proof of Theorem~\ref{RankOne}.} 
The function $\Bell$ is a generalized quasi convex function, hence it satisfies \eqref{BellmanInequality} for every $\varphi\in C_0^{\infty}([0,1]^d)$. Let us fix $x\in\mathbb{R}^d, \lambda \in \mathbb{R}$. We plug $\lambda l_{x,\eps,\delta}$ into \eqref{BellmanInequality}. We get (for every $e\in E$)
\begin{equation*}
\begin{split}
\Bell(e) \leq \int\limits_{[0,1]^d} \Bell(e + \gengrad [\lambda l_{x,\eps,\delta}]) 
&= \int\limits_{B} \Bell(e+ \gengrad[\lambda l_{x,\eps,\delta}]) + \int\limits_{[0,1]^d\backslash B} \Bell(e + \gengrad [\lambda l_{x,\eps,\delta}]) 
\\&\leq \int\limits_{B} \Bell(e+ \gengrad[\lambda l_{x,\eps,\delta}]) + O\big(\lambda (\|e\|+\|e_x\| +\eps)\delta\big)
\end{split}
\end{equation*}
from Lemma \ref{Laminate}.
Since~$\gengrad [l_{x,\eps,\delta}]\big|_{B}$ is equimeasurable ($B$ equipped with the measure~$\frac{dx}{|B|}$) with~$\cos(2\pi t)e_x$, 
\begin{equation*}
\int\limits_{B} \Bell(e+ \gengrad[\lambda  l_{x,\eps,\delta}])\frac{dx}{|B|}= \int\limits_{[0,1]} \Bell(e+ \lambda \cos(2\pi t) e_x)dt.
\end{equation*}
Therefore,
\begin{equation*}
\Bell(e) \leq |B| \int\limits_{[0,1]} \Bell(e+ \lambda \cos(2\pi t) e_x)dt + O\big(\lambda (\|e\|+\|e_x\| +\eps)\delta\big).
\end{equation*}
Since for $\delta\to 0$, we have $|B|\to 1$, and then
\begin{equation}\label{rankonecos}
 \Bell(e) \leq \int\limits_{[0,1]} \Bell(e+ \lambda \cos(2\pi t) e_x)dt.
\end{equation}
For a fixed $e\in E$, consider the function $\mathbb{R} \ni s \mapsto \Bell(e + s e_x) $. By~\eqref{rankonecos},
\begin{equation*}
 \Bell(e+ s e_x)\leq  \int\limits_{[0,1]}\Bell(e+s e_x +\lambda \cos(2\pi t) e_x)dt.
\end{equation*}
Thus, by Lemma \ref{ConvexityWithCosine}, the function~$\mathbb{R} \ni s \mapsto \Bell(e + s e_x)$ is convex (one simply applies lemma to this function).  Since $e\in E$ and $x\in \mathbb{R}^d$,~$\lambda \in \mathbb{R}$ were arbitrary, it proves the generalized rank one convexity of the function $\Bell$.
\qed

\section{Separately convex homogeneous functions and proof of Theorem \ref{MainTheoremBis}}\label{S4}
\begin{Le}\label{Obvious}
Generalized rank one vectors span~$E$.
\end{Le}
\begin{proof}
Since $E$ is a finite dimensional Hilbert space, every functional on $E$ is of the form $\phi^*(\cdot)= \langle \sum_{\alpha\in A} a_{\alpha} e_{\alpha},\;\cdot\;\rangle$.  We get
\begin{equation*}
 \phi^*(e_x)= \sum_{\alpha\in A} a_{\alpha} x^{\alpha} i^{|\alpha|+|\alpha_0|}
\end{equation*}
for every $x\in\mathbb{R}^d$.
If $E$ is not a span of generalized rank one vectors, then there exists a non trivial $\phi^*$ such that
\begin{equation*}
 0=\phi^*(e_x)=\sum_{\alpha\in A} a_{\alpha} x^{\alpha} i^{|\alpha|+|\alpha_0|}
\end{equation*}
for every $x\in\mathbb{R}^d$. However, $x^{\alpha}$ are linearly independent monomials. Therefore, $a_{\alpha}=0$ for every $\alpha\in A$. Hence $\phi^*\equiv 0$ and the generalized rank one vectors span $E$.
\end{proof}
We recall that our aim was to show that~$T_1$ is a linear combination of the other~$T_j$. By comparing the kernels of the~$\tilde{T}_j$, its is equivalent to the fact that~$V \geq 0$ everywhere. By the evident inequality~$B \leq V$, it suffices to prove that~$B$ is non-negative. By Lemma~\ref{Obvious} and Theorem~\ref{RankOne}, this will follow from the theorem below. Hence it suffices to prove Theorem \ref{SeparateTheorem}  to get Theorem \ref{MainTheoremBis}.
\begin{Def}\label{SeparatelyConvex}
A function~$F: \mathbb{R}^d \to \mathbb{R}$ is separately convex if it is convex with respect to each variable.
\end{Def}
\begin{Th}\label{SeparateTheorem}
A function~$F:\mathbb{R}^d \to \mathbb{R}$ that is separately convex and positively homogeneous of order one is non-negative
\end{Th}
Before passing to the proof, we cite Theorem~$2.31$ of the book~\cite{Dacorogna}, which says that a separately convex function  is continuous. This fact will be implicitly used several times in the reasoning below. 
\begin{proof}
We proceed by induction. Suppose that the statement of the theorem holds true for the dimension~$d-1$, we then prove it for the dimension~$d$. Construct the function~$G:\mathbb{R}^{d-1}\to\mathbb{R}$ by the formula
\begin{equation*}
G(x) = F(x,1),\quad x \in \mathbb{R}^{d-1}.
\end{equation*}
This function is separately convex and convex with respect to radius, i.e. for every $x\in \mathbb{R}^{d-1}$ the function $\mathbb{R}_{+}\ni t \mapsto G(tx)$ is a convex function. Indeed, the function $F$ is positively one homogeneous and separately convex, thus for $t,r>0$ and $\tau\in (0,1)$ we have:
\begin{equation*}
\begin{split}
 \tau G(tx) &+(1-\tau) G(rx)= \tau F( tx ,1) +(1-\tau) F(rx,1)
 =(\tau t + (1-\tau) r) \left( \frac{\tau t F( x ,\frac{ 1}{ t}) +(1-\tau)r F(x,\frac{1}{r})}{ \tau t + (1-\tau) r }\right)
 \\&\geq (\tau t + (1-\tau) r) F\left ( x, \frac{1}{\tau t + (1-\tau) r }\right)
 = F\left((\tau t + (1-\tau) r)x, 1\right) = G\left( (\tau t + (1-\tau) r)x\right).
\end{split}
 \end{equation*}
We claim that for each~$x\in\mathbb{R}^{d-1}$ the function~$\mathbb{R}\ni t \mapsto G(tx)$ is convex. Since the function~$G$ is continuous, it suffices to prove that~$G(tx) + G(-tx) \geq G(0)$ for all~$t \in \mathbb{R}$. Consider another function~$V$:
\begin{equation*}
V(x) = \lim_{t\to 0+}\frac{G(tx) + G(-tx)- 2G(0)}{t},\quad x \in\mathbb{R}^{d-1}.
\end{equation*}
The limit exists due to the convexity with respect to radius. This function~$V$ is positively one homogeneous and separately convex.  However, it may have attained the value~$-\infty$. Fortunately, this is not the case. If there exists $x\in\mathbb{R}^d$ such that $V(x)=-\infty$ then the following holds:
\begin{equation*}
 2 V(0,x_2,\ldots, x_d ) \leq V(x_1, \ldots , x_d) + V(-x_1, \ldots, x_d)=-\infty.
\end{equation*}
Therefore $V(0,x_2,\ldots, x_d )=-\infty$. We repeat the above reasoning with $x_2, \ldots, x_d$ instead of $x_1$ and we get that $V(0)=-\infty$, but from the definition of $V$ we know that
\begin{equation*}
 V(0)=\lim_{t\to 0+}\frac{G(0) + G(0)- 2G(0)}{t}=0.
\end{equation*}
Hence $V(x)$ is finite for every $x\in\mathbb{R}^{d-1}$. Thus, by the induction hypothesis,~$V$ is non-negative. So,~$\mathbb{R}\ni t \mapsto G(tx)$ is a convex function.

By symmetry,~$G(x) + G(-x) \geq 2F(x,0)$. On the other hand,~$\lim_{t \to \pm\infty}\frac{G(tx)}{t} = F(x,0)$. So, the convexity of $t\mapsto G(tx)$ gives the inequality~$|G(x) - G(-x)| \leq 2F(x,0)$. Adding these two inequalities, we get that~$F(x,1) \geq 0$.    
\end{proof}
\paragraph{Proof of Theorem~\ref{MainTheoremBis}.} Assume that inequality~\eqref{MainNonInequality} holds. Then, by Theorem~\ref{PropertiesBell}, the function~$\Bell$ given by~\eqref{Bellman} is Lipshitz, positively one homogeneous, generalized quasi convex, and satisfies the inequality~$B \leq V$, where the function~$V$ is given by formula~\eqref{FunctionV}. Then, by Theorem~\ref{RankOne},~$\Bell$ is a generalized rank one convex function. 

Let~$e \in E$ be an arbitrary point. By Lemma~\ref{Obvious},~$e$ is a linear combination of generalized rank one vectors~$e_{x_1},e_{x_2},\ldots,e_{x_k}$. We may assume that they are linearly independent. Consider the function~$F: \mathbb{R}^k \to \mathbb{R}$ given by the rule
\begin{equation*}
F(z_1,z_2,\ldots,z_k) = \Bell(z_1e_{x_1} + z_2e_{x_2} + \ldots + z_ke_{x_k}).
\end{equation*}
By the generalized rank one convexity of~$\Bell$,~$F$ is separately convex. It is also positively one homogeneous, thus~$F \geq 0$ by Theorem~\ref{SeparateTheorem}. Therefore,~$\Bell(e)$ is also non-negative for arbitrary~$e \in E$. 

Since~$\Bell \geq 0$, we have~$V \geq 0$. In such a case, it follows from formula~\eqref{FunctionV} that~$\Ker \tilde{T}_1 \supset \cap_{j = 2}^{\ell} \Ker \tilde{T}_j$. Therefore,~$T_1$ is a linear combination of the other~$T_j$.
\qed

\section{Related questions}\label{S5}
\paragraph{Towards Conjecture~\ref{MainTheorem}.} The following statement plays the same role in view of Conjecture~\ref{MainTheorem}, as Theorem~\ref{SeparateTheorem} plays in the proof of Theorem~\ref{MainTheoremBis}.
\begin{Cj}\label{Harmonic}
Let~$F:\mathbb{R}^{2d} \to \mathbb{R}$ be a Lipschitz positively homogeneous function of order one. Suppose that for any~$j = 1,2,\ldots,d$ the function~$F$ is subharmonic with respect to the variables~$(x_j,x_{j+d})$. Then\textup,~$F$ is non-negative.
\end{Cj}
Indeed, plugging the cosine function into~\eqref{BellmanInequality} as we did in the proof of Theorem~\ref{RankOne} leads to ``subharmonicity''\footnote{The ``subharmonicity'' means that~$D\Bell \geq 0$ as a distribution, where~$D$ is an elliptic symmetric differential operator of second order (with constant real coefficients); one can then pass to usual subharmonicity by an appropriate change of variable.} of the function~$\Bell$ in the directions of projections of a generalized rank one vector onto subspaces generated by odd and even monomials in~$A$ correspondingly. Therefore, Conjecture~\ref{MainTheorem} follows from Conjecture~\ref{Harmonic}.

We are not able to prove Conjecture~\ref{Harmonic}. However, we know the following: in the case~$d=1$, the function~$F$ is not only non-negative, but, in fact, convex (i.e. a positively one homogeneous subharmonic function is convex). On the other hand, there is no much hope for simplifications: a subharmonic positively one homogeneous function in~$\mathbb{R}^3$ (and thus in~$\mathbb{R}^d$,~$d \geq 3$) can attain negative values, e.g. in~$\mathbb{R}^4$ one may take the function~$\frac{x_1^2 + x_2^2 + x_3^2 - x_4^2}{\sqrt{x_1^2 + x_2^2 + x_3^2}}$. 

There are also reasons that differ from the ones discussed in the present paper that may ``break'' inequality~\eqref{MainNonInequality}. One of them is  a certain geometric  property of the spaces generated by the operators~$T_j$. Not stating any general theorem or conjecture, we treat an instructive example. Consider the non-inequality
\begin{equation}\label{example}
\|\partial_1^{2}\partial_2 f\|_{L_1} \lesssim \|\partial_1^{4} f\|_{L_1} + \|\partial_2^2 f\|_{L_1}.
\end{equation} 
Conjecture~\ref{MainTheorem} hints us that it cannot be true. We will disprove it on the torus~$\mathbb{T}^2$ and leave to the reader the rigorous formulation and proof of the corresponding transference principle, whose heuristic form is ``inequalities of the sort~\eqref{MainNonInequality} are true or untrue simultaneously on the torus and the Euclidean space''. Consider two anisotropic homogeneous Sobolev spaces~$W_1$ and~$W_2$, which are obtained from the set of trigonometric polynomials by completion and factorization over the null-space with respect to the seminorms
\begin{equation*}
\|f\|_{W_1} = \|\partial_1^{4} f\|_{L_1} + \|\partial_2^2 f\|_{L_1}, \quad \|f\|_{W_2} = \|\partial_1^{2}\partial_2 f\|_{L_1} + \|\partial_1^{4} f\|_{L_1} + \|\partial_2^2 f\|_{L_1}.
\end{equation*}
If inequality~\eqref{example} holds true, then these two spaces are, in fact, equal (the identity operator is a Banach space isomorphism between these spaces). However, it follows from the results of~\cite{PW} (see~\cite{W2,W3} as well) that~$W_2$ has a complemented translation-invariant Hilbert subspace\footnote{That means that there exists a subspace~$X \subset W_2$ such that~$g \in X$ whenever~$g(\cdot + t) \in X$,~$t \in \mathbb{T}^2$,~$X$ is isomorphic to an infinite dimensional Hilbert space, and there exists a continuous projector~$P: W_2 \to X$.}, whereas~$W_1$
does not, a contradiction.

\paragraph{Martingale transforms.}
Let~$S = \{S_n\}_n$,~$n \in \{0\}\cup \mathbb{N}$, be an increasing filtration of finite algebras on the standard probability space. We suppose that it differentiates~$L_1$ (i.e. for any~$f \in L_1(\Omega)$ the sequence~$\E(f\mid S_n)$ tends to~$f$ almost surely). We will be working with martingales adapted to this filtration. 
\begin{Def}\label{MartingaleTransform}
Let~$\alpha = \{\alpha_n\}_n$ be a bounded sequence. The linear operator
\begin{equation*}
T_{\alpha}[f] = \sum\limits_{j=1}^{\infty} \alpha_{j-1}(f_{j} - f_{j-1}), \quad f = \{f_n\}_n\;\hbox{is an~$L_1$ martingale},
\end{equation*}
is called a martingale transform.  
\end{Def}
Our definition is not as general as the usual one, and we refer the reader to the book~\cite{Osekowski} for the information about such type operators. We only mention that martingale transforms serve as a probabilistic analog for the Calder\'on--Zygmund operators. The probabilistic version of Conjecture~\ref{MainTheorem} looks like this.
\begin{Cj}\label{MartingaleConjecture}
Suppose~$\alpha^1,\alpha^2,\ldots,\alpha^{\ell}$ are bounded sequences. Suppose that the algebras~$S_n$ unifromly grow\textup, i.e. there exists~$\gamma < 1$ such that each atom~$a$ of~$S_n$ is split in~$S_{n+1}$ into atoms of probability not greater than~$\gamma|a|$ each. The inequality
\begin{equation}\label{MartingaleInequality}
\|T_{\alpha^1}f\|_{L_1} \lesssim \sum\limits_{j=2}^{\ell}\|T_{\alpha^j}f\|_{L_1}
\end{equation}
holds for any martingale~$f$ adapted to~$\{S_n\}_n$ if and only if~$\alpha^1$ is a sum of a linear combination of the~$\alpha^j$ and an~$\ell_1$ sequence. 
\end{Cj}
We do not know whether the condition of uniform growth fits this conjecture. Anyway, it is clear that one should require some condition of this sort (otherwise one may take~$S_n=S_{n+1}=\ldots = S_{n+k}$ very often and loose all the control of the sequences~$\alpha^j$ on this time intervals). Again, we are not able to prove the conjecture in the full generality, but will deal with an important particular case.
\begin{Th}\label{MartingaleTheorem}
Suppose~$\alpha^1,\alpha^2,\ldots,\alpha^{\ell}$ to be bounded periodic sequences. The inequality
\begin{equation*}
\|T_{\alpha^1}f\|_{L_1} \lesssim \sum\limits_{j=2}^{\ell}\|T_{\alpha^j}f\|_{L_1}
\end{equation*}
holds if and only if~$\alpha^1$ is a linear combination of the other~$\alpha^j$. 
\end{Th}
\begin{proof}
To avoid technicalities, we will be working with finite martingales (denote the class of such martingales by~$\M$). The general case can be derived by stopping time. Assume that inequality~\eqref{MartingaleInequality} holds true. Consider the Bellman function~$\Bell: \mathbb{R}^{\ell} \to \mathbb{R}$ given by the formula
\begin{equation*}
\Bell(x) = \inf_{f\in \M}\Big(\sum\limits_{j=2}^{\ell}\big\|x_j + T_{\alpha^j}[f]\big\|_{L_1} - c\big\|x_1 + T_{\alpha^1}[f]\big\|_{L_1}\Big).
\end{equation*}
It is easy to verify that this function is positively one homogeneous and Lipschitz. Moreover,~$\Bell$ is convex in the direction of~$(\alpha^1_n,\alpha^2_n,\ldots,\alpha^{\ell}_n)$ for each~$n$ (by the assumption of periodicity, there is only a finite number of these vectors); the proof of this assertion is a simplification of Theorem~\ref{RankOne} (here we do not have to make additional approximations; however, see~\cite{SZ}, Lemma~$2.17$ for a very similar reasoning). Thus, by Theorem~\ref{SeparateTheorem},~$\Bell$ is non-negative on the span of~$\{(\alpha^1_n,\alpha^2_n,\ldots,\alpha^{\ell}_n)\}_n$. Since~$\Bell(x) \leq \sum_{j \geq 2}|x_j| - c|x_1|$, the aforementioned span does not contain the~$x_1$-axis. Therefore,~$\alpha^1$ is a linear combination of the other~$\alpha^j$. 
\end{proof}

\paragraph{Case~$p>1$.} Inequality~\eqref{MainNonInequality} may become valid provided one replaces the~$L_1$-norm with the~$L_p$ one,~$1 < p < \infty$. Let~$c_p$ be the best possible constant in the inequality
\begin{equation}\label{MainNonInequalityForP}
\|T_1 f\|_{L_p(\mathbb{R}^d)}^p \leq c_p\sum\limits_{j = 2}^{\ell}\|T_j f\|_{L_p(\mathbb{R}^d)}^p.
\end{equation}
It is interesting to compute the asymptotics of~$c_p$ as~$p \to 1$. Some particular cases have been considered in~\cite{BBPW}, we also refer the reader there for a discussion of similar questions.  
\begin{Cj}\label{AsympConj}
Let~$\Lambda$ be a pattern of homogeneity in~$\mathbb{R}^d$\textup, let~$\{T_j\}_{j=1}^{\ell}$ be a collection of~$\Lambda$-homogeneous differential operators. If~$T_1$ cannot be expressed as a linear combination of the other~$T_j$\textup, then~$c_p \gtrsim \frac{1}{p-1}$.
\end{Cj}  
The conjecture claims that if there is no continuity at the endpoint, then the inequality behaves at least as if it had a weak type~$(1,1)$ there (it is also interesting to study when there is a weak type~$(1,1)$ indeed). First, we note that this question is interesting even when there are only two polynomials. Second, this is only a bound from below for~$c_p$. Even in the case of two polynomials,~$c_p$ can be as big as~$(p-1)^{1-d}$ (and thus the endpoint inequality may not be of weak type~$(1,1)$, at least when~$d \geq 3$), see~\cite{BBPW} for the example.

As in the previous point, Conjecture~\ref{AsympConj} will follow from the corresponding geometric statement in the spirit of Theorem~\ref{SeparateTheorem}.
\begin{Cj}\label{SeparateTheoremP}
Let~$F: \mathbb{R}^d \to \mathbb{R}$ be a Lipschitz separately convex~$p$-positively homogeneous function \textup(i.e.\textup,~$F(\lambda x) = |\lambda|^pF(x)$\textup). Suppose that~$F(x) \leq |x|^p$. Then\textup,~$F(x) \gtrsim (1-p)|x|^p$.
\end{Cj}
Conjecture~\ref{AsympConj} is derived from Conjecture~\ref{SeparateTheoremP} in the same way as Theorem~\ref{MainTheoremBis} derived from Theorem~\ref{SeparateTheorem}: one considers the Bellman function~\eqref{Bellman} with the function~$V$ given by the formula
\begin{equation*}
V(e) = \Big(c_p\sum\limits_{j = 2}^{\ell}|\tilde{T}_j e|^p - |\tilde{T}_1 e|^p\Big), 
\end{equation*}
proves its generalized quasi convexity, which leads to the generalized rank one convexity, and then uses Conjecture~\ref{SeparateTheoremP} to estimate~$c_p$ from below.

It is not difficult to verify the case~$d = 2$ of Conjecture~\ref{SeparateTheoremP}. Therefore, there exists a~$C_0^{\infty}$-function~$f_p$ such that
\begin{equation*}
(p-1)\|\partial_1\partial_2f_p\|_{L_p(\mathbb{R}^2)} \gtrsim \Big(\|\partial_1^2f_p\|_{L_p(\mathbb{R}^2)} + \|\partial_2^2f_p\|_{L_p(\mathbb{R}^2)}\Big).
\end{equation*}

Krystian Kazaniecki

Institute of Mathematics, University of Warsaw.

\medskip

krystian dot kazaniecki at mimuw dot edu dot pl.

\bigskip
 
Dmitriy M. Stolyarov

Institute for Mathematics, Polish Academy of Sciences, Warsaw;

P. L. Chebyshev Research Laboratory, St. Petersburg State University;

St. Petersburg Department of Steklov Mathematical Institute, Russian Academy of Sciences (PDMI RAS).

\medskip

dms at pdmi dot ras dot ru, dstolyarov at impan dot pl.

\medskip

http://www.chebyshev.spb.ru/DmitriyStolyarov.

\bigskip

Michal Wojciechowski

Institute of Mathematics, Polish Academy of Sciences;

Institute of Mathematics, University of Warsaw.

\medskip

m dot wojciechowski at impan dot pl.

\end{document}